\documentclass[11pt, english]{article}
 \pdfoutput=1
\usepackage{hyperref}
\usepackage[utf8]{inputenc}
\usepackage{amsxtra}
\usepackage{color}
\usepackage{amsopn}
\usepackage{amsmath,amsthm,amssymb,amscd}
\usepackage{mathtools}
\usepackage{eufrak}
\usepackage{tikz}
\usepackage{tikz-cd}
\usepackage{calrsfs}
\usepackage[all,cmtip]{xy}
\usepackage{footmisc}

\theoremstyle{remark}
\newtheorem{remark}{Remark}[section]

\theoremstyle{definition}
\newtheorem{defn}[remark]{Definition}
\newtheorem{ex}[remark]{Example}
\newtheorem{rem}[remark]{Remark}

\theoremstyle{theorem}
\newtheorem{thm}[remark]{Theorem}
\newtheorem{prop}[remark]{Proposition}
\newtheorem{corol}[remark]{Corollary}
\newtheorem{lemma}[remark]{Lemma}

\newcommand{\lie}[1]{\mathfrak{#1}}

\newcommand{\G}{G_2}

\newcommand{\beq}{\begin{equation}}
\newcommand{\eeq}{\end{equation}}
\newcommand{\bqn}{\begin{eqnarray}}
\newcommand{\eqn}{\end{eqnarray}}
\newcommand{\bqne}{\begin{eqnarray*}}
\newcommand{\eqne}{\end{eqnarray*}}

\newcommand{\R}{{\mathbb R}}

\newcommand{\C}{{\mathbb C}}
\newcommand{\Z}{\mathbb Z}

\newcommand{\SU}{{\rm SU}}
\newcommand{\SO}{{\rm SO}}

\newcommand{\U}{\mathrm{U}}

\newcommand{\su}{\mathfrak{su}}

\newcommand{\so}{\mathfrak{so}}

\newcommand{\gt}{\mathfrak{g}}
\newcommand{\tg}{\mathfrak{t}}

\newcommand{\RP}{\mathbb{RP}}
\newcommand{\Ft}{\mathcal{F}}
\newcommand{\GL}{\mathrm{GL}(7,\R)}
\newcommand{\gl}{\mathfrak{gl}(7,\R)}
\newcommand{\hg}{\mathfrak{h}}
\newcommand{\Pt}{\mathcal{P}}
\newcommand{\It}{\mathcal{I}}

\newcommand{\Spin}{\mathrm{Spin}}
\newcommand{\spin}{\mathfrak{spin}}

\usepackage[affil-it]{authblk}

\begin{document} 
\title{{Non orientable three-submanifolds of $\G-$manifolds}}
\author{Leonardo Bagaglini\thanks{Dipartimento di Matematica e Informatica {``}Ulisse Dini{"}. Università degli Studi di Firenze. Viale Morgagni, 67/a - 50134, Firenze, Italy. E-mail address: leonardo.bagaglini@unifi.it.}}
\maketitle
\abstract{By analogy with associative and co-associative cases we introduce a class of three-dimensional non-orientable submanifolds, of almost $\G-$manifolds, modelled on planes lying in a special $\G-$orbit. An application of the Cartan-K\"ahler theory shows that some three-manifold can be presented in this way. We also classify all the homogeneous ones in $\RP^7$.}

\section*{Introduction}
The compact, simply connected, exceptional, real, Lie group $\G$, as subgroup of $\SO(7)$, acts on the grassmannians $\mathrm{Gr}_k(\R^7)$ of $k-$planes. This action in transitive unless $k=3,4$ (see \cite{Bry2}). It happens that, in $k=3,4$ cases, the action has cohomogeneity one with principal isotropy a reducible representation of $\SO(3)$. In the oriented case, the remaining two special orbits are isomorphic and have isotropy representation given by a reducible, eight dimensional, $\SO(4)-$module; these planes correspond to associative and co-associative ones. But if we consider the not-oriented case, a different special orbit arises, with isotropy type $\SO(3)\times\Z_2$. Its planes are the only ones reversed by $\G$. Therefore it is natural to ask whether there exist not-orientable submanifolds, of almost $\G-$manifolds, modelled on that orbit. The answer is positive, indeed we prove that such class of manifolds is rich in examples and shows interesting properties.\par\medskip
The paper is structured as follows. First, in Section \S\ref{sec1}, we analyse the action of $\G$ on grassmannians, proving the structure theorem of the orbit space. In Section \S\ref{sec2} we introduce the main definition of \emph{$\varphi-$planes}, which determines our local model. Next, in Section \S\ref{sec3}, we recall fundamental concepts of Cartan-K\"ahler theory and in Section \S\ref{sec4} we prove that some closed, analytic, three-manifold can be presented as \emph{$\varphi-$}submanifold of an open $\G-$manifold (Theorem \ref{thm} and Corollary \ref{cor1}). In Section \S\ref{sec5} we classify all the homogeneous not-orientable $\varphi-$manifolds in $\RP^7$, when equipped with the canonical nearly parallel structure (Theorem \eqref{homogphiman}).\par\medskip
Generally we refer to manifolds equipped with a $\G-$structure as $\G-$manifolds in the torsion-free case, otherwise as \emph{almost} $\G-$manifolds.

\section{$\G$ actions on Grassmannians}\label{sec1}
In this section we describe the action of $\G$ on grassmannians of three and four-planes of $\R^7$, both oriented and non oriented. In the sequel if $Z$ is a compact manifold on which a compact Lie group $G$ acts with cohomegenity one, by notation $Z/G=[G/H_1|G/K|G/H_2]$ we mean that the orbits space $Z/G$ is diffeomorphic to the closed interval and $G/H_1$, $G/H_2$, $G/K$ are models for special orbits and generic one respectively. We refer to \cite{Bre} for an exhaustive treatment of general theory of compact Lie groups. 
\par\medskip
Let $X$ be the grassmannian $\mathrm{Gr}^+_4(\R^7)$ of oriented four-planes of $\R^7$. We often denote planes $\xi\in X$ by $u_1\wedge u_2\wedge u_3\wedge u_4$, for ordered linear independent vectors $u_1,u_2,u_3,u_4\in\xi$. Identify $\R^7$ with the reducible $\SO(4)-$module 
\begin{equation}\label{mod1}
\quad\R^4\oplus\Lambda^2_-\R^4,
\end{equation}
where $\Lambda^2_-\R^4$ represents the pre-dual space of anti-self dual two-forms on $\R^4$.\\ 
Let $(x^1,x^2,x^3,r^0)$ be standard coordinates on $\R^4$ so that the two-forms
$$\omega_1=dr^{0}\wedge dx^1+dx^2\wedge dx^3,\quad\omega_2=dr^{0}\wedge dx^2-dx^1\wedge dx^3,\quad\omega_3=dr^{0}\wedge dx^3+dx^1\wedge dx^2,$$
are a basis of $\Lambda^2_-(\R^4)^*$. Consider $(r^1,r^2,r^3)$ the dual basis of $(\omega_1,\omega_1,\omega_3)$ in $\Lambda^2_-\R^4$. Then the three-form
\begin{equation}\label{stdcoo}
\varphi=\omega_1\wedge dr^1+\omega_2\wedge dr^2+\omega_3\wedge dr^3-dr^{123},
\end{equation}
is stable\footnote{Its $\GL-$orbit is open.} and positive\footnote{It induces a positive definite metric.} (see \cite{Hit}), therefore the action of its stabiliser, in $\GL$, on $\R^7$  corresponds to the seven dimensional irreducible representation $\G\subset\SO(7),$ where $(\underline{x},\underline{r})$ are orthonormal and positive oriented coordinates. We refer to $\varphi$ as the \emph{fundamental} three-form associated to $\G$\footnote{Such form is defined by (the irreducible seven dimensional representation of) $\G$ up to constant. It can be fixed by choosing an orientation and imposing $||\varphi||^2=7$.}.
\begin{defn}
We call  $(\underline{x},\underline{r})$ \emph{Cayley coordinates} if $\varphi$ is given by $\eqref{stdcoo}$.
\end{defn}
\begin{rem}
Observe that the Hodge dual $\phi$ of $\varphi$, also known as the \emph{fundamental} four-form, is given by
$$\phi=-\omega_1\wedge dr^{12}+\omega_2\wedge dr^{13}-\omega_3\wedge dr^{12} +dr^0\wedge dx^1 \wedge dx^2 \wedge dx^3.$$
with respect to Cayley coordinates.
\end{rem}
\begin{prop}\cite{Fri}
The action of $\G$ on $X$ has cohomogeneity one with principal isotropy $\SO(3)\subset\SO(4)$, via representation \eqref{mod1}. There are two $\G$-isomorphic special orbits, each one of isotropy type $\SO(4)$, which are through
$$\xi_+=\frac{\partial}{\partial r^0}\wedge\frac{\partial}{\partial x^1}\wedge\frac{\partial}{\partial x^2}\wedge\frac{\partial}{\partial x^3}\quad\text{and}\quad \xi_-=-\xi_+,$$
respectively.
\end{prop}
\par
A path $\xi$, parametrizing the orbits, starting from $\xi_-$ and ending at $\xi_+$, is given by
\begin{gather*}
\xi_\theta=\left(\mathrm{sin}(\theta)\frac{\partial}{\partial r^0}+\mathrm{cos}(\theta)\wedge \frac{\partial}{\partial r^1}\right)\wedge\frac{\partial}{\partial x^1}\wedge\frac{\partial}{\partial x^2}\wedge\frac{\partial}{\partial x^3},\quad \forall \theta\in\left[-\frac{\pi}{2},\frac{\pi}{2}\right].
\end{gather*}
Observe that 
$$\phi|_{\xi_\theta}=\mathrm{sin}(\theta)(\mathrm{sin}(\theta)dr^0+\mathrm{cos}(\theta)dr^1)\wedge dx\wedge dy\wedge dz,$$
thus there exists only one (principal) orbit $\mathcal{O}_0$, through $\xi_0$, such that 
$$\mathcal{O}_0=\left\lbrace\sigma\in X\;|\;\phi|_{\sigma}=0\right\rbrace.$$
Hence the following holds.
\begin{corol}
$\G$ can reverse planes lying in $\mathcal{O}_0$.
\end{corol}
\begin{proof}
The corollary follows since condition $\phi|_\sigma=0$, if $\sigma\in X$, does not depend on the orientation, hence both $\sigma$ and $-\sigma$ belong to the same orbit.
\end{proof}
 \par
Let 
\begin{equation*}
\varepsilon:X\longrightarrow X,
\end{equation*}
be the map which reverses the orientations, and 
$$p:X\longrightarrow Y,$$
be the quotient map over $Y=X/\varepsilon$, the grassmannian $\mathrm{Gr}_4(\R^7)$ of non oriented planes.\\
Then 
$$\varepsilon(\mathcal{O}_{\pm})=\mathcal{O}_{\mp}\quad\text{and}\quad\varepsilon(\mathcal{O}_0)=\mathcal{O}_0.$$
\begin{prop}
The action of $\G$ on $Y$ has cohomogeneity one with principal isotropy $\SO(3)$. There are two, not equivalent, special orbits, of isotropy types $\SO(4)$ and $\SO(3)\times\mathbb{Z}_2$ respectively.
\end{prop}
\begin{proof}
It is easy to see that there exists an open, connected,  $\G-$invariant neighbourhood $\mathcal{V}$ of $\mathcal{O}_+$ such that $\varepsilon(\mathcal{V})\cap\mathcal{V}=\emptyset$.  Hence the restriction of $p$
$$p|_\mathcal{V}:\mathcal{V}\longrightarrow p(\mathcal{V}),$$ is a $\G-$equivariant diffeomorphism onto an open $\G-$invariant neighbourhood of $p(\mathcal{O}_+)$. Then it follows that the action has cohomogeneity one, principal isotropy $\SO(3)$ and one special orbit $p(\mathcal{O_+})$ of isotropy type $\SO(4)$. Now let $H$ and $K$ be the stabilisers of $p(\xi_0)\in Y$ and $\xi_0\in X$ respectively. Since 
$$hKh^{-1}\subseteq K,\quad\forall h\in H,$$ 
and $p$ is a double cover, $K$ is a normal subgroup of $H$ with index $2$. Thus $H=K \times\mathbb{Z}_2$ is the isotropy of the special orbit $p(\mathcal{O}_0)$ through $p(\xi_0)$. Summarizing
$$Y/\G=\left[\G/SO(4)|\G/\SO(3)|\G/\SO(3)\times\mathbb{Z}_2\right].$$
\end{proof}\medskip
The analysis we have just completed turns out to describe also the orbits space of both $\mathrm{Gr}_3^+(\R^7)$ and $\mathrm{Gr}_3(\R^7)$, as shown by the following remark.  
\begin{rem}
Since
$$\mathrm{Gr}_4^+(\R^7)\cong\mathrm{Gr}_3^+(\R^7)\quad\text{and}\quad \mathrm{Gr}_4(\R^7)\cong\mathrm{Gr}_3(\R^7),$$
by $\SO(7)-$equivariant isomorphisms, such manifolds are isomorphic as $\G-$spaces too.
\end{rem}

\par\bigskip
\section{Local model}\label{sec2}
In this section we investigate the local model of three-planes lying in $\mathcal{O}_0$.\medskip\par 
By analogy with the previous section let 
$$p:\mathrm{Gr}_3^+(\R^7)\rightarrow \mathrm{Gr}_3(\R^7),$$
be the two-fold covering which forgets the orientations.\\
\begin{defn}
Planes lying in $p(\mathcal{O}_0)\subset \mathrm{Gr}_3(\R^7)$ are called $\varphi-$planes.
\end{defn}
\begin{rem}
$\varphi-$planes are characterized by the the vanishing of the three-form $\varphi$. In fact the path
$$\xi_{\theta}=\left(\mathrm{sin}(\theta)\frac{\partial}{\partial t^1}+\mathrm{cos}(\theta)\frac{\partial}{\partial x^1}\right)\wedge\frac{\partial}{\partial x^2}\wedge\frac{\partial}{\partial x^3},\quad \theta\in\left[-\frac{\pi}{2},\frac{\pi}{2}\right],$$
parametrizes the orbits and meets $\mathcal{O}_0$ in $\theta=0$.
\end{rem}
Now, consider Cayley coordinates $(\underline{x},\underline{r})$ and let $\xi=\frac{\partial}{\partial x^1}\wedge\frac{\partial}{\partial x^2}\wedge\frac{\partial}{\partial x^3}\in \mathcal{O}_0$. Its stabiliser in $\G$ is  
\begin{equation*}
K=\left\{\left(\begin{matrix}
a&0&0\\
0&1&0\\
0&0&a\\
\end{matrix}\right)\;|\;a\in\SO(3)\right\}\cong \SO(3).
\end{equation*}
While the stabiliser of its projection $p(\xi)$ is $H=K\times \Z_2,$ where
 \begin{equation*}
\Z_2=\left\lbrace\left(\begin{matrix}
1_{\SO(3)}&0&0\\
0&1&0\\
0&0&1_{\SO(3)}\\
\end{matrix}\right),\left(\begin{matrix}
-1_{\SO(3)}&0&0\\
0&-1&0\\
0&0&1_{\SO(3)}\\
\end{matrix}\right)\right\rbrace.
\end{equation*}
Consider $\xi^\perp$, explicitly $\xi^{\perp}=\frac{\partial}{\partial r^0}\wedge\frac{\partial}{\partial r^1}\wedge\frac{\partial}{\partial r^2}\wedge\frac{\partial}{\partial r^3},$
and decompose it as orthogonal sum $\xi^\perp=\mu\oplus\lambda,$ where $\mu=\mathrm{Span}(\frac{\partial}{\partial r^0})$ and $\lambda=\mathrm{Span}(\frac{\partial}{\partial r^1},\frac{\partial}{\partial r^2},\frac{\partial}{\partial r^3})$. Let $H_\xi$ and $H_{\xi^{\perp}}$ be the groups $\mathrm{O}(\mu)\times\SO(\lambda)$ and $\mathrm{O}(\xi^{\perp})$ respectively.
Then the following proposition holds.
\begin{prop}\label{rapp}
There exist omomorphisms $f_1$ and $f_2$ such that the diagrams
\begin{equation*}
\begin{tikzcd}
H\arrow{r}{}\arrow{d}{f_1}&\mathrm{GL}(\xi)\\
H_\xi
\arrow{ru}{}&
\end{tikzcd}
\quad\text{and}\quad
\begin{tikzcd}
H\arrow{r}{}\arrow{d}{f_2}&\mathrm{GL}(\xi^\perp)\\
H_{\xi^\perp}\arrow{ru}{}&
\end{tikzcd}
\end{equation*}are commutative.
Moreover the representations of $H$ on $\xi^{\perp}=\mu\oplus\lambda$ and $\Lambda^3\xi^*\oplus\Lambda^2\xi^*$ are equivalent, with an explicit equivariant isomorphism given by
\begin{equation*}
\begin{CD}
\mu\oplus\lambda@>>>\Lambda^3\xi^*\oplus\Lambda^2\xi^*,\\
(u,v)@>>> (i_u\phi|_\xi,i_v\varphi|_\xi).
\end{CD}
\end{equation*}
\end{prop}
\begin{proof}
The proof is a straightforward computation performed in Cayley coordinates.
\end{proof}
\section{Cartan-K\"ahler theory}\label{sec3}
For the sake of clarity we recall the key concepts and theorems of Cartan-K\"ahler theory and its applications to $\G-$geometry. For a detailed treatment we refer to \cite{BryB} and \cite{Bry}.\par\medskip
Let $N$ be a manifold and $\It$ be a differential ideal of the ring $\Omega^*(N)$. Denote by $\It^n$ the intersection $\It\cap\Omega^n(N)$ and suppose that $\It^0$ is empty.\\
Let $E$ be a $n-$dimensional subspace of some tangent space, $T_xN$, of $N$. We say $E$ an \emph{integral element} of $\It$, equivalently $E\in V^n(\It)$, if every $n-$form lying in $\It$ vanishes when restricted to $E$. An integral element is said to be \emph{ordinary} if, locally,  $V^n(\It)\subset\mathrm{Gr}_n(TN)$ appears as the zero locus of some non zero functions with linear independent differentials.\\
If $E$ is an integral element of $\It$ we define its polar space $H(E)$ as the set of $n+1-$dimensional integral extensions of $E$, explicitly 
$$H(E)=\left\{v\in T_xX\,|\,(i_v\delta)|_E=0,\;\forall \delta\in\It^{n+1}\right\}.$$ 
The \emph{extension rank} of $E$ is the integer $r(E)=\mathrm{dim}H(E)-n-1$. Observe that $E$ is maximal if and only if $r(E)=-1$.
We say $E$ \emph{regular} if it is ordinary and $r$ is locally constant around it.\\
An integral manifold $Y$ of  $\It$ is a submanifold of $N$ whose tangent spaces are all integral elements. It is said to be ordinary, or regular, if its tangent spaces are. Now we are ready to state the Cartan-K\"ahler Theorem. 
\begin{thm}[Cartan-K\"ahler Theorem]
Let $N$ be an analytic manifold, $\It\subset\Omega^*(N)$ be a analytic, differential ideal and $X$ be one of its analytic, integral, $n-$dimensional manifolds. Suppose $X$ is regular, with extension rank $r\geq 0$, and let $Z$ be an analytic submanifold, of codimension $r$, containing $X$ and transversal to each of its polar spaces.\\
Then there exists an analytic, integral, $(n+1)-$dimensional manifold $Y$ satisfying $$X\subset Y\subset Z.$$
Moreover, if $Y'$ is a manifold with the same properties, then $Y\cap Y'$ is still an integral $(n+1)-$dimensional manifold.
\end{thm}
In order to verify regularity of an integral element we will need the following result, known as Cartan's test of regularity. But first we give some other definitions.\par
Let $E$ be a $n-$dimensional, integral element of $\It$. An integral flag $(E_j)_j$, of length $n$ with terminus $E$, is an increasing filtration of $n+1$ vector spaces verifying the followings:
$$E_0\subset E_1\subset\dots\subset E_n= E,\quad E_j\in V^j(\It),\quad \mathrm{dim}E_j=j\quad j=0,\dots,n.$$
Let $(E_j)_j$ be an integral flag of length $n$ and $c_j$ be the codimension of $H(E_j)$ in the appropriate tangent space, for each $j$. Call $(E_j)_j$ regular if $E_j$, $j<n$, is regular and denote by $C$ the sum over $j$ of each $c_j$. 
\begin{prop}[Cartan's test]
Let $(E_j)_j$ be a regular integral flag of length $n$.\\
Then, locally around $E_n$, $V^n(\It)$ lies in a codimension $C$ submanifold of the grassmannian $\mathrm{Gr}_n(TN)$. Moreover the flag is regular if and only if, near $E_n$, $V^n(\It)$ is a smooth manifold of codimension $C$.
\end{prop}\par\medskip
Let $N$ be a seven dimensional manifold and suppose its frame bundle\footnote{The action of $a\in\GL$ on $\xi\in\Ft$ is given by $\xi.a=\xi\circ a$.} $\Ft\xrightarrow{P}N$, where 
$$\Ft=\left\{\xi_x\;|\;\xi_x:\R^7\rightarrow T_{x}N,\;x\in N\right\},$$
can be reduced to a principal $\G-$bundle $\Pt$. This is equivalent to the existence of a global section of the fiber bundle $\Ft/\G\xrightarrow{f} N$ defined by the diagram
\begin{equation*}
\begin{tikzcd}
\Ft\arrow{r}{Q}\arrow{dr}{P}&\Ft/\G\arrow{d}{f}\\
&N\\
\end{tikzcd}
\end{equation*}
 or, similarly, to the existence of a stable and positive (locally given by \eqref{stdcoo}), three-form $\varphi$.\par
Let us denote the total space $\Ft/\G$ by $S$. Following Bryant (see \cite{Bry}) we introduce differential ideals $\It$'s on $\Ft$ and $S$ (labelled by the same letter), related by the projection $Q$, as follows. Let $\Lambda^*(\R^7)^{\G}$ be the ring of $\G-$invariant, constant coefficients, differential forms on $\R^7$, and, for each $\delta_0\in\Lambda^*(\R^7)^{\G}$, consider
$$\tilde{\delta}\in\Omega^*(\Ft),\quad\text{where}\quad \delta_0\left(\xi^*\left(P_*(.)\right)\right)=\tilde{\delta}_{\xi}(.),\quad\forall \xi\in\Ft.$$
Now define
$$\It=<\left\{ d(\tilde{\delta})\;|\;\delta_0\in\Lambda^*(\R^7)^{\G}\right\}>.$$ \par
The role of $\It$ in the study of $\G-$structures is well explained by the following theorem. Recall that a $\G$-structure is said to be \emph{torsion-free} if the Levi-Civita connection restricts to $\Pt$.\footnote{Other equivalent conditions are: $\varphi$ is parallel; $\varphi$ is both closed and co-closed}.  
\begin{thm}
Let $V^7(\It,f)$ be the set of seven dimensional integral elements of $\It$ which are transversal to the fibers of $f:S\rightarrow N$. Then $V^7(\It,f)$ consists of tangent spaces to graphs of local sections corresponding to torsion-free structures. 
\end{thm}
\begin{rem}\label{rem1}
Let $F$ be the projection $\Ft\rightarrow N$. Then $$\mathrm{codim}(V^7(\It,F),\mathrm{Gr}_7(T\Ft))=\mathrm{codim}(V^7(\It,f),\mathrm{Gr}_7(TS)).$$
\end{rem}

\section{Existence results}\label{sec4}
In this section we prove existence of closed, connected, three-submanifolds of (open) $\G-$manifolds, modelled on $\varphi-$planes.\par\medskip
\begin{defn}
A submanifold $X$ of an almost $\G-$manifold $(N,\varphi)$ is said to be \emph{$\varphi-$}manifold if all its tangent spaces are $\varphi-$planes or, equivalently, if the pullback of $\varphi$ to $X$ vanishes.   
\end{defn}
\begin{rem}
Any submanifold $X$ of a co-associative one, say $Y$, is, by definition, $\varphi-$manifold. Moreover, since  for any $\varphi-$plane there is a unique direction defining a co-associative extension, such $X$ defines $TY|_X$.
\end{rem}

Before proving the main theorem of this section we need the following lemma.
\begin{lemma}\label{id}
Let $X$ be a $\varphi-$manifold of some, connected, almost $\G$ manifold $(N,\varphi)$. Assume that there exist an isometry $\tau$ of $N$ and a point $x\in X$ such that $$\tau(x)=x,\quad (\tau^*)\varphi_x=\varphi_x,\quad \tau^*|_{T_x^*X}=\mathrm{Id}_{T_x^*X}.$$ Then $\tau=\mathrm{Id}_N$. In particular if $\tau_1$ and $\tau_2$ are isometries of $N$ preserving the structure, and they agree on some open set of $X$, then $\tau_1=\tau_2$.
\end{lemma}
\begin{proof}
Consider Cayley coordinates around $x$. Since $\tau^*_x$ preserves $\varphi_x$ its matrix representation lies in $\G$; explicitly
$$\tau^*_x=\left(\begin{matrix}
\pm t&0&0\\
0&\pm 1&0\\
0&0&t
\end{matrix} \right),\quad\text{for some}\quad t\in\SO(3).$$
But $\pm t$ represents $\tau^*_x|_{T_x^*X}$, which is the identity, hence $\tau^*_x=\mathrm{Id}_{T_xN}$. Since an isometry fixing $x$ and acting identically on $T_xN$ must be $\mathrm{Id}_N$ ($N$ is connected), it follows $\tau=\mathrm{Id}_N$.  
\end{proof}
\begin{rem}
Recall that, as stressed by Bryant in \cite{Bry}, a closed, orientable, analytic, Riemannian three-manifold always admits an analytic parallelization. This follows since every orientable three-manifold is parallelizable (see \cite{WU}) and analytic differential forms on closed, analytic manifolds, are dense in the space of smooth differential forms (see \cite{Bo}).
\end{rem}
\begin{thm}\label{thm}
Let $X$ be a closed, connected, orientable, analytic, Riemannian, three-manifold. Then $X$ can be isometrically embedded into an open, analytic, $\G-$manifold $N$ as $\varphi-$manifold contained in a co-associative one, the last isometric to $X\times S^1$. Moreover if there exists an analytic, non trivial, involutive isometry  $\tau\in \mathrm{Iso}(X)$ and an orthonormal co-frame $(\alpha_1,\alpha_2,\alpha_3)$ such that $\tau^*\alpha_j=\alpha_j$, if $j=1,2$, and $\tau^*\alpha_3=-\alpha_3$, then $\tau$ can be extended to an unique involutive isometry, which preserves the structure, on whole $N$. 
\end{thm}
\begin{proof}
Let $\tau$ be as in the second part of the statement, otherwise put $\tau=\mathrm{Id}_X$. Let $\eta=(\alpha_1,\alpha_2,\alpha_3)$ be an orthonormal global co-frame and $f\in\mathrm{O}(3)$ defined by $$f.\eta =(\tau^*\eta).$$
Consider the section $\xi$ of the the frame bundle $\Ft$, over $M=X\times S^1\times \R^3$, defined by the co-frame $(\alpha_1,\alpha_2,\alpha_3,dr^0,dr^1,dr^2,dr^3)$,  where $r^0$ is the angle coordinate on $S^1$ and $(r^1,r^2,r^3)$ are coordinates on $\R^3$. Now define $\underline{\tau}\in\mathrm{Diff}(M)$ as 
$$\underline{\tau}(x,r^0,r^1,r^2,r^3)=(\tau(x),\mathrm{det}(f)r^0,(\mathrm{det}(f)f).(r^1,r^2,r^3)),\quad\forall (x,\underline{r})\in M.$$
Thanks to the following identification of principal bundles
\begin{equation*}
\begin{CD}
M\times\GL@>>>\Ft,\\
(x,\underline{r};a)@>>>\xi_{(x,\underline{r})}.a.
\end{CD}
\end{equation*}
the action of $\underline{\tau}$ on $\Ft$ is given by
$$\underline{\tau}_*(x,\underline{r};a)=(\underline{\tau}(x,\underline{r});ta),\quad\forall (x,\underline{r})\in M,$$
where $t=t^{-1}\in\G$ is 
$$
t=\left(
\begin{matrix}
f&0&0\\
0&\mathrm{det}(f)&0\\
0&0&\mathrm{det}(f)f\\
\end{matrix}
\right).
$$
In fact, if $m=(x,\underline{r})\in M$ and $a\in\GL$, the following diagram turns out to be commutative
\begin{equation*}
\begin{tikzcd}
\R^7\arrow{rr}{a}&&\R^7\arrow{rr}{\xi_{m}}\arrow{dd}{t_m}\arrow{ddrr}{(\tau_*\xi)_{\tau(m)}}&&T_{m}M\arrow{dd}{\tau_*}\\
&&&&\\
&&\R^7\arrow{rr}{\xi_{\tau(m)}}&&T_{\tau(m)}M,
\end{tikzcd}
\end{equation*}
hence $\tau_*(\xi_m .a)=\xi_{\tau(m)}.t_m a$.\\
Let $S$ be the total space of the $\G-$structure bundle associated to $\Ft$ defined by
$$Q:\Ft\longrightarrow \Ft/\G=S.$$
The map $\underline{\tau}_*$ descends to a well defined diffeomorphism $[\underline{\tau}_*]$ of $S$, given by\footnote{In the sequel square brackets denote points in $S$ as follows $\left[x,\underline{r};a\right]=\left\{(x,\underline{r};ab)\;|\;b\in\G\right\}$.}
$$[\underline{\tau}_*][x,\underline{r};a]=[\underline{\tau}(x,\underline{r});ta].$$
Now, the form $\tilde{\varphi}\in \Omega^3(M)$
$$\tilde{\varphi}=(dr^0\wedge \alpha_1+\alpha_2\wedge\alpha_3)dr^1+(dr^0\wedge \alpha_2-\alpha_1\wedge\alpha_3)dr^2+(dr^0\wedge \alpha_3+\alpha_1\wedge\alpha_2)dr^3-dr^1\wedge dr^2\wedge dr^3,$$
is stable, positive, $\underline{\tau}-$invariant and related to $\sigma\in\Gamma(M,S)$
$$\sigma(x,\underline{r})=[x,\underline{r};1],\quad \forall (x,\underline{r})\in M.$$
Let $\underline{\varphi}$ be its pullback on $\Ft$.
\par
Before proceeding further fix $(x,\underline{r})\in M$. Observe that $f$ leaves unchanged a complete flag of $\R^3$
$$\left\{0\right\}=L_0\subset L_1\subset L_2\subset L_3=\R^3.$$
Then we may consider the complete flag $(F_k)_k$ of $\R^7$ given by
$$\left\{0\right\}=\R\subset \R^2\subset\R^3\subset \R^4\subset\R^4\oplus L_1\subset  \R^4\oplus L_2\subset \R^4\oplus L_3=\R^7.$$
Consider a seven dimensional integral element $E_7\subset T\Ft$ of the ideal $\It$, intrudced in \S\ref{sec3}, transverse to the fiber over $(x,\underline{r})$. If $\theta=(\theta_k)_k$ represents the tautological one-form on $\Ft$ with respect to $\xi$\footnote{Explicitly $\theta_k(\xi^h)=\delta_k^h$.}, then $E_7$ is the terminus of the complete integral flag $(E_k)_k$ given by
\begin{equation*}
\begin{cases}
E_k=\left\{e\in E_7\,|\,\theta_j(e)=0,\;k<j \right\},&\text{if}\quad 0\leq k\leq 4,\\
E_5=\left\{e\in E_7\,|\,\theta_5(e)+\theta_6(e)+\theta_7(e)\in L_1\right\},\\
E_6=\left\{e\in E_7\,|\,\theta_5(e)+\theta_6(e)+\theta_7(e)\in L_2\right\}.
\end{cases}
\end{equation*}\par
In order to compute the polar spaces of $E_k$ identify $\R^k$ with $F_k$, define, for $0\leq k\leq 7$, $\iota_k:\R^k\rightarrow\R^7$, and consider the decreasing filtration $(\hg_k)_k$ of vector spaces given by 
$$\mathfrak{\hg}_k=\left\{A\in\gl\;|\;\iota_k^*(A^*\delta)=0\quad\forall \delta\in{\left(\Lambda^* \R^7\right)}^{\G} \right\}.$$
Observe that, if $c_k=\mathrm{codim}(\hg_k,\gl)$, it turns out that
$$(c_0,\dots,c_7)=(0,0,0,1,5,15,28,35).$$ 
In fact the computation is straightforward for $k<5$, and, if $k\geq 5$, there exists a $\G-$isometry turning $F_k$ into any given $k-$plane: if we choose one of them a computation shows the claim.\par
Then, identifying $\GL$ with the fiber, it turns out 
$$H(E_k)=\hg_k+E_7.$$
Thus, by Cartan's test, $(E_k)_k$ is a regular integral flag of $\It$. Moreover, by Remark \ref{rem1}, also $([E_k])_k$ is an integral regular flag of $\It$ on $S$, still transverse to the fiber.\par
For future reference observe that, since $\iota_k(\R^k)$, for $k\geq 4$, is a $t-$module, $t$ acts on $\R^k$ and 
$$\iota_k^*t^*=t^*\iota_k^*,\quad\text{on}\quad \Lambda^*\R^7.$$
As consequence, if $k\geq 4$, $\hg_k$ turns out to be $\mathrm{Ad}(t)-$invariant:
$$\iota_k^*((tAt)^*\delta)=\iota_k^*(t^{*}A^*t^*\delta)=t ^{*}\iota_k^*(A^*\delta)=0,\quad\forall A\in\hg_k,\delta\in{\left(\Lambda^* \R^7\right)}^{\G}.$$
By hypotheses we can equip $\gl$ with an $\mathrm{Ad}(t)-$invariant metric and define the increasing filtration $(W_k)_k$ of invariant subspaces given by 
$$W_k=\hg_k^{\perp},\quad k\geq 4.$$
Now, for each $k\geq 4$, let $U_k$ be an $\mathrm{Ad}(t)-$invariant open neighbourhood of $0\in W_k$ such that the map
\begin{equation*}\begin{CD}
U_k\times\G@>>>\GL,\\
(u,g)@>>> e^ug,
\end{CD}
\end{equation*}
is an embedding. It exists since $W_k$ does not intersect $\hg_k$, which contains $\gt$. With no loss of generality we may suppose $U_k\subset U_{k+1}$.\par
Finally we are ready to apply the Cartan-K\"aheler Theorem to produce integral manifolds of $\It$.\\
First, define $X_4$ as 
$$X_4=\left\{[x,r^0,\underline{0};1]\;|\;x\in X,\;r^0\in\R\right\rbrace.$$
Obviously $X_4$ is $[\underline{\tau}_*]-$invariant. Moreover it is a $4-$dimensional integral manifold of $\It$ (since $(Q_*\underline{\varphi})|_{X_4}=0$ and $d(Q_*\underline{\phi})|_{X_4}=0$) whose tangent spaces are all regular elements of type $[E_4]$ with respect to a regular flag introduced before. Consequently $X_4$ has extension rank $r(X_4)=32$.\\
Now consider the $10-$dimensional manifold
$$Z_4=\left\lbrace [x,r^0,s;e^u]\;|\;x\in X,\;r^0\in\R,\;s\in L_1,\;u\in U^4\right\rbrace.$$
$Z_4$ is $[\underline{\tau}_*]-$invariant, in fact 
\begin{align*}
[\underline{\tau}_*]\left[x,{r^0},s;e^u\right]=&\left[\underline{\tau}(x,{r^0},s);te^u \right]\\
=&\left[\underline{\tau}(x,{r^0},s);e^vt\right],\\
=&\left[\underline{\tau}(x,r^0,s);e^v\right]\\
&\text{for some $v\in U^4$}.
\end{align*}
Moreover its tangent spaces are transversal to the polar spaces of type $H([E_4])$, and, by the Cartan-K\"ahler theorem, there exists a $5-$dimensional integral manifold $Y_4$ of $\It$, verifying 
$$X_4\subset Y_4\subset Z_4.$$
By invariance of $Z_4$ and uniqueness also $X_5=Y_4\cap [\underline{\tau_*}]Y_4$ is a $5-$dimensional integral manifold of $\It$ with the same property.\\
 Replacing $X_5$ with a neighbourhood of $X_4$ if necessary we may assume that $X_5$ is connected, $[\underline{\tau_*}]-$invariant, and the graph of a section over an open neighbourhood of $\left\{(x,r^0,0,0,0)\right\}$ in $\left\{(x,r^0,s)\,|\,s\in L_1\right\}$. Since it is a graph, its tangent spaces are regular of type $E_5$.\par
Now it would be clear what strategy we are following. The rest of the proof will proceed as we have just seen.\par
Define the $21-$dimensional manifold $Z_5$ as follows
$$Z_5=\left\lbrace [x,{r^0},s;e^u]\;|\;x\in X,\;r^0,\in\R,\;s\in L_2,\;u\in U^5\right\rbrace.$$
It is $[\underline{\tau_*}]-$invariant, meets the polar spaces of $X_5$ transversally, and its codimension equals the extension rank of $X_5$.\\
Thus there exists a $6-$dimensional integral manifold $Y_5$ of $\It$ satisfying 
$$X_5\subset Y_5\subset Z_5.$$
Defining $X_6$ to be $Y_5\cap [\underline{\tau_*}] Y_5$, and replacing it with a suitable neighbourhood of $X_5$, we may assume that $X_6$ is a connected graph of a section over an open neighbourhood of $\left\{(x,{r^0},s)\,|\,s\in L_1\right\}$ in $\left\{(x,{ r^0},s)\,|\,s\in L_2\right\}$, hence also regular.\par
Define the $35-$dimensional manifold $Z_6$ as follows
$$Z_6=\left\lbrace [x, r^0,\underline{r};e^u]\;|\;x\in X,\;r^0,r^1,r^2,r^3\in\R,\;u\in U^6\right\rbrace.$$
It is $[\underline{\tau_*}]-$invariant, meets the polar spaces of $X_6$ transversally, and its codimension equals the extension rank of $X_6$.\\
Thus there exists a $7-$dimensional integral manifold $Y_6$ of $\It$ satisfying 
$$X_6\subset Y_6\subset Z_6.$$
Finally defining $Y$ as $Y_6\cap [\underline{\tau_*}]Y_6$, and replacing it with a suitable neighbourhood of $X_6$, we may assume that $Y$ is a connected graph of a section, say $\varsigma$, over an open ($\underline{\tau}-$invariant) neighbourhood $N$ of $\left\{(x, r^0,s)\,|\,s\in L_2\right\}$ in $M$. Such $\varsigma$ defines a torsion-free $\G-$structure on $N$, which agrees with that one induced by ${\varphi}$ on $X$. Hence $X$ turns out to be a $\varphi-$manifolds contained in the, compact, co-associative, $X\times S^1$. Finally the restriction of $\underline{\tau}$ is the unique isometry, by Lemma \ref{id}, which extends $\tau$. 
\end{proof}
The previous theorem allows us to prove the following corollary, on existence of non-orientable $\varphi-$manifolds.
\begin{corol}\label{cor1}
Let $X$ be a closed, connected, non orientable, analytic, Riemannian three-manifold, and $\pi:X'\rightarrow X$ its Riemannian orientation covering. Suppose there exist two orthonormal one-forms, $a_1$ and $a_2$, on $X$. Then there exist two open $\G-$manifolds $N'$ and $N$ containing $X'$ and $X$ as $\varphi-$manifolds respectively, and a two-fold covering $\underline{\pi}:N'\rightarrow N$ preserving the structures and extending $\pi$. Moreover $X'$ and $X$ are contained in $\underline{\pi}$-related, closed, co-associative submanifolds $Y'$, $Y$, isometric to $X'\times S^1$ and $(X'\times S^1)/\Z_2$ respectively. In particular $X$ defines a non trivial class of $H^1(Y,\Z_2)$ (see \cite{BW}). 
\end{corol}
\begin{proof}
Let $\tau$ the non trivial deck transformation of $\pi$. Fix an orientation and a $\tau-$invariant metric on $X'$ and define $\alpha_j=\pi^*a_j$, for $j=1,2$, and $\alpha_3=*(\alpha_1\wedge\alpha_2)$. Obviously $(\alpha_1,\alpha_2,\alpha_3)$ satisfies the hypothesis of Theorem \ref{thm}, therefore there exists an open $\G-$manifold $N'$, containing $X'$ as $\varphi-$manifold (inducing the same Riemannian structure). Moreover there exists a unique isometry $\underline{\tau}$, extending $\tau$ and preserving the structure. By Lemma \ref{id}  such extension verifies $\underline{\tau}^2=\mathrm{Id}_M$.\par
Since the group generated by $\underline{\tau}$ acts freely we can consider a $\underline{\tau}-$invariant tubular neighbourhood $N(X'\times S^1)$, of $X'\times S^1$ in $N'$, on which the restriction of $\underline{\tau}$ has no fixed points. Consequently the space $N=N(X'\times S^1)/\underline{\tau}$ turns out to be a manifold. Furthermore $N$ inherits a torsion-free $\G-$structure and its submanifold $X'/\underline{\tau}$, naturally isometric to $X$, satisfies the condition of being a $\varphi-$manifold. Observe that the last is contained in the compact co-associative submanifold $(X'\times S^1)/\underline{\tau}$. 
\end{proof}\par
The following example shows a manifold obtained with trivial applications of the Cartan-K\"ahler argument.
\begin{ex}\label{flat torus}
Let $X'=\R^3/\Z^3$ be the three-torus and $\tau$ be the involution
$$\tau(x^1,x^2,x^3)=\left(x^2,x^1,x^3+\frac{1}{2}\right), \quad \forall(x^1,x^2,x^3)\in X'.$$
Then $\tau$ is the non trivial deck transformation of an orientation covering $\pi:X'\rightarrow X$, since it has no fixed points.
Now consider $\left\{dx^1,dx^2,dx^3\right\}$ on $X'$ and let $T^4$ be the four-torus\footnote{We could use $\R^4$ as well.} and $f\in\mathrm{Diff}(T^4)$, $f_\tau\in\mathrm{Diff}(X'\times T^4)$ given by
$$f(r^0,r^1,r^2,r^3)=(-r^0,-r^2,-r^1,-r^3),\quad f_\tau=\tau\times f.$$
Then the form 
\begin{eqnarray*}\tilde{\varphi}=(dr^0\wedge dx^1+dx^2\wedge dx^3)dr^1+(dr^0\wedge dx^2-dx^1\wedge dx^3)dr^2+\nonumber\\(dr^0\wedge dx^3+dx^1\wedge dx^2)dr^3-dr^1\wedge dr^2\wedge dr^3.
\end{eqnarray*}
descends to a stable, positive, closed and co-closed, three-form $\varphi$ on $M=(X'\times T^4)/f_\tau$.\par
Observe that the submanifolds
\begin{equation*}
\begin{cases}
X_4=\left(X'\times T^1\times\left\{(0,0,0)\right\}\right)/f_\tau,\\
X_5=\left(X'\times T^1\times\left\{(r,r,0)\,|\, r\in \R\right\}\right)/f_\tau,\\
X_6=\left(X'\times T^3\times\left\{0\right\}\right)/f_\tau,\\
\end{cases}
\end{equation*}
satisfy
$$X\subset X_4\subset X_5 \subset X_6 \subset M.$$

\end{ex}
The next example shows that co-associative manifolds also arise in non trivial torsion classes of $\G-$structures.
\begin{ex}\label{Ex2}
Consider $\su(2)$ spanned by the Pauli matrices $\sigma_1,\sigma_2,\sigma_3$. Since the constant structures are $2\epsilon_{ijk}$ (sign of $(ijk)\in\mathbb{S}_3$) the isomorphism 
$$
f_*=
\left(\begin{matrix}
0&0&-1\\
0&-1&0\\
-1&0&0\\
\end{matrix}\right),
$$
lies in $\mathrm{Aut}(\su(2))$. Hence it defines an (involutive) automorphism $f$ of $\SU(2)$. Denote by $r_1,r_2,r_3$ the left-invariant one-forms defined by the generators of $\su(2)$. Now let $S^1$ be the unit circle equipped with the angle coordinate $r^0$ and define $f_0\in \mathrm{Diff}(S^1)$ as $f_0(e^{2\pi i r^0})=e^{-2\pi i r^0}$. Obviously $f_0^*dr^0=-dr^0$. \par
If $X$ is a closed, non orientable, three-manifold let $X'\xrightarrow{\pi}X$ be its orientation covering with not trivial deck transformation $\tau$. Suppose there is a global co-frame $(\alpha_1,\alpha_2,\alpha_3)$ verifying
$$\tau^*\alpha_1=\alpha_3,\quad\tau^*\alpha_2=\alpha_2.$$
Now, on $M'=X'\times S^1\times \SU(2)$, define the stable and positive three-form
\begin{eqnarray*}\tilde{\varphi}=(dr^0\wedge \alpha_1+\alpha_2\wedge\alpha_3)r_1+(dr^0\wedge \alpha_2-\alpha_1\wedge\alpha_3)r_2+(dr^0\wedge \alpha_3+\alpha_1\wedge\alpha_2)r_3+\\\nonumber-r_1\wedge r_2\wedge r_3.
\end{eqnarray*}
Such $\tilde{\varphi}$ is invariant under the action of the involutive, with no fixed points, diffeomorphism $$f_\tau=\tau\times f_0\times f\in\mathrm{Diff}(M'),$$ so that it defines a stable and positive three-form $\varphi$ on the compact manifold $M=M'/f_\tau$. Moreover, identifying $X$ with the submanifold $(X'\times\left\{1\right\}\times\left\{1_{\SU(2)}\right\})/f_\tau$ or  $(X'\times\left\{-1\right\}\times\left\{1_{\SU(2)}\right\})/f_\tau$, it turns out that $\varphi|_X=0$.\par
Now if we consider $X'$ to be the flat torus and $\alpha_1,\alpha_2,\alpha_3$ as in Example \ref{flat torus},  defining\footnote{Here, if $\delta\in\Omega^*(M')$ is invariant under $f_\tau^*$, then $[\delta]$ denotes the correspondent differential form on $M$.} $$2\chi_3=[r_1\wedge r_2\wedge r_3]\quad\text{hence}\quad 2*\chi_3=[-dr^0\wedge\alpha_1\wedge\alpha_2\wedge\alpha_3],$$ the $\G-$structure defined above satisfies
$$
\begin{cases}
d{\varphi}=\frac{1}{2}{\phi}+*\chi_3,\\
d\phi=0.
\end{cases}
$$
In particular the submanifold $Y=(X'\times S^1\times\left\{1_{\SU(2)}\right\})/f_\tau$ is calibrated by $\phi$, thus it is volume-minimizing in its homological class.
\end{ex}

\section{An homogeneous classification}\label{sec5}
In this section we classify all the $G-$homogeneous not-orientable $\varphi-$manifolds of $\RP^7$, where $G$ is a closed subgroup of $\Spin(7)$ and $\RP^7$ is equipped with the canonical nearly parallel $\G-$structure.\par\medskip 
\begin{defn}
Let $G$ be a closed and connected subgroup of $\Spin(7).$ We refer to a three-dimensional $G-$homogeneous $\varphi-$manifold $X$ of $\RP^7$, or $S^7$, as \emph{homogenous $\varphi-$manifold}.
\end{defn}
\begin{rem}
In the previous sections we have seen examples of $\varphi-$manifolds arising as three-submanifolds of some co-associative ambient. In $\RP^7$ co-associative manifolds does not exist, therefore, in this setting, $\varphi-$manifolds cannot be extended, neither locally, to co-associative ones.
\end{rem}
We are able to prove the following theorem.
\begin{thm}\label{homogphiman}
There exists a non-orientable, homogeneous, $\varphi-$manifold $X$ in $\RP^7$. Moreover any other submanifold sharing the same properties can be turned into $X$ by an element of $\Spin(7)$.
\end{thm}
First let us recall how the nearly parallel $\G-$structure is defined. Let $\R^8$ be the fundamental $\Spin(7)-$module and $\Phi_0\in\Lambda^4(\R^8)^*$ the $\Spin(7)-$invariant form given by
$$\Phi_0=\frac{1}{2}\omega_0^2+\Re(\Psi_0),$$
where $2i\omega_0=dz_1d\bar{z_1}+dz_2d\bar{z_2}+dz_3d\bar{z_3}+dz_4d\bar{z_4}$ and $\Psi_0=dz_1dz_2dz_3dz_4$ with respect to complex coordinates $(z_1,z_2,z_3,z_4)$ of $\C^4=\R^8$. 
Let us identify $\R^8\setminus\left\{0\right\}$ with $\R^*\times S^7$ via
$$\R^*\times S^7\ni(r,x)\mapsto rx\in\R^8\setminus\left\{0\right\}.$$
Then there exists a unique $\varphi\in\Omega^3(S^7)$ such that
$$\Phi_0=r^3dr\wedge\varphi+r^4*_7\varphi.$$
Such $\varphi$ defines a $\Spin(7)-$invariant, nearly parallel ($d\varphi=4*_7\varphi$), $\G-$structure on $S^7$, therefore on $\RP^7$. For a more exhaustive treatment see \cite{Lot}.
\begin{rem}
Clearly any $X\subset\RP^7$ arises from some $\Z_2-$invariant $X'\subset S^7$. Therefore the classification of not-orientable homogeneous $\varphi-$manifolds in $\RP^7$ reduces to the classification of homogeneous $\varphi-$manifolds in $S^7$ on which $-1\in\Z_2$ acts as an orientation reversing map. By abuse of notation we continue to denote $X'$ by $X$.
\end{rem}
The following result, which {easily} follows from standard representation theory, represents the main tool to prove Theorem \eqref{homogphiman}.
\begin{prop}\label{subgroups}
If $K$ is a three-dimensional closed and connected subgroup of $\Spin(7)$ then it is conjugated to one of the following, listed together with their defining representations on $\R^8$ and the dimensions $d$ of their centralizers\footnote{In the following if $W$ is a complex $K-$module then $[W]$ denotes the real module equal to $W$ whereas $[[W]]$ denotes the real module such that $\C\otimes[[W]]=W$.},
\begin{enumerate}
\item $\U(1)^3\subset\SU(4)$ via $[\C^4]$ with $d=0$;
\item $\SU(2)\subset\SU(4)$ via $[S^3\C^2]$ with $d=1$;
\item $\SU(2)\subset\SU(4)$ via $[\C^2\oplus\C^2]$ with $d=6$;
\item $\SU(2)\subset\SU(3)\subset\G$ via $[\C^2]\oplus\R\oplus\R\oplus\R\oplus\R$ with $d=6$;
\item $\SU(2)\subset\SO(4)\subset\G$ via $[[S^2\C^2]]\oplus[\C^2]\oplus\R$ with $d=3$;
\item $\SO(3)\subset\SU(3)\subset\G$ via $\R^3\oplus\R^3\oplus\R\oplus\R$ with $d=1$;
\item $\SO(3)\subset\G$ via $[[S^6\C^2]]\oplus\R$ with $d=0$.
\end{enumerate}
\end{prop}
\begin{proof}
Clearly, by compactness, either $K$ is abelian or it is locally isomorphic to $\SU(2)$. In the first case $K$ is a maximal torus of $\Spin(7)$ and the proposition follows, so let us assume the second holds.\par
Firstly we prove that the above representations are the only admissible ones. Denote by $\lie{k}$ the Lie algebra of $K$, which is isomorphic to $\su(2)$. The only irreducible real $\lie{k}-$modules of dimension less than or equal to $8$ are $V_k$ for $k=\mathrm{dim}(V_k)=1,3,4,5,7$ or $8$, and the following are not allowed:
\begin{enumerate}
\item[(i)] $V_5\oplus V_3$;
\item[(ii)] $V_5\oplus V_1\oplus V_1\oplus V_1$;
\item[(iii)] $V_3\oplus V_1\oplus V_1\oplus V_1\oplus V_1\oplus V_1$.
\end{enumerate}\par
Indeed if one of these occurred, $\lie{k}$ should be contained in the algebra of a stabiliser of a three-plane of $\R^8$, which is isomorphic to $\so(4)=\su(2)\oplus\su(2)$ contained in some $\gt_2$. But the only three-dimensional simple subalgebras of $\so(4)$, up to inner automorphism, are three not-equivalent $\so(3)$: namely $\left\{(X,0)\right\}$, $\left\{(0,X)\right\}$ and $\left\{(X,X)\right\}$  with representations on the correspondent $\R^7$ given by $\R^3\oplus[\C^2]$, $\R^3\oplus\R\oplus\R\oplus\R\oplus\R$ and $\R^3\oplus\R^3\oplus\R$ respectively.\par
Now we prove that any previous representation is given by a unique, up to $\Spin(7)-$action, Lie subgroup. This essentially follows from Dynkin's classification of $\mathfrak{sl}(2,\C)$ subalgebras of classical simple Lie algebras (see \cite{OniVil} and the references therein).\par
Clearly if $K$ represents a maximal torus the claim is straightforward. Even if $K$ lies in some $\G$ there is nothing to prove: any subgroup of $\G$, which is locally isomorphic to $\SU(2)$, must be conjugated to one of those listed above (see also \cite{Ma}), and all the subgroups of $\Spin(7)$ isomorphic to $\G$ are conjugated.\par
Let us suppose that $K$ corresponds to one of the remain representations, that are (2) and (3). Clearly $-\mathrm{id}_{\R^8}$ belongs to $K$, therefore the projection of $K$ into $\SO(7)$, say $H$, is a copy of $\SO(3)$. Since the only irreducible real $\SO(3)-$modules of dimension less than or equal to $7$ are $V_k$ with $k$ equal to $1,3,5$ or $7$ there are few admissible type-decompositions of $\R^7$. Moreover the restriction of the standard metric on each factor coincide, up to a scalar multiple, to the unique $H$-invariant metric, since the modules are of real type.\par 
One can observe that $V_3\oplus V_3\oplus V_1$ and $V_7$ are modules realized by subgroups of some $\G\subset\SO(7)$ and thus they cannot occur in our hypothesis, but this is not relevant. In each case $H$ must be cojugated, by an element of $\mathrm{O}(7)$, to the standard $\SO(3)$ defined by the correspondent representation; actually in $\SO(7)$, combing with $-\mathrm{id}_{\R^7}$. Therefore also $K$ will be conjugated, by an element of $\Spin(7)$, to the appropriate $\SU(2)$.\par
Finally let $\lie{z}$ be the Lie algebra of the centralizer of $K$ in $\Spin(7)$.\par
In case 1 $d=0$.\par
In case 2 it is well known that $\spin(7)=\su(4)\oplus\lie{m}$, where $\lie{m}\cong[[\Lambda^2\C^4]]$ as $\SU(4)-$modules. Therefore $$\spin(7)=\underbrace{\su(2)\oplus\lie{n}}_{\su(4)}\oplus[[S^4\C^2]]\oplus[[\C]],$$ with $\mathrm{Ad}(\SU(2))\lie{n}=\lie{n}$, and consequently $\R\cong\lie{z}$.\par
For the following cases let us recall that $\spin(7)=\gt_2\oplus\R^7$ as $\gt_2-$modules and that $\gt_2=\su(3)\oplus[\C^3]$ as $\SU(3)-$modules.\par
In case 3 
$$\spin(7)=\underbrace{\su(2)\oplus\R^3\oplus [S^2\C^2]\oplus[\C]\oplus\R}_{\su(4)}\oplus [\C^2\oplus\C^2\oplus\C]\oplus[[\C]],$$
and thus $\lie{z}\cong\R^6$.\par
In case 4
$$\spin(7)=\underbrace{\underbrace{\su(2)\oplus[\C^2]\oplus\R}_{\su(3)}\oplus[\C^2]\oplus\R\oplus\R}_{\gt_2}\oplus[\C^2]\oplus\R\oplus\R\oplus\R,$$
and consequently $\lie{z}\cong\R^6$.\par
In case 5
$$\spin(7)=\underbrace{\underbrace{\su(2)\oplus\R\oplus\R\oplus\R}_{\so(4)}\oplus[\C^2\oplus \C^2]}_{\gt_2}\oplus\R^3\oplus[\C^2],$$
which means $\lie{z}\cong\R^3$.\par
In case 6
$$\spin(7)=\underbrace{\underbrace{\so(3)\oplus\R^5}_{\su(3)}\oplus\R^3\oplus\R^3}_{\gt_2}\oplus\R^3\oplus\R^3\oplus\R,$$
which gives $\lie{z}\cong\R^1$.\par
Finally, in case 7, a direct computation shows that $d=0$.
\end{proof}
\begin{lemma}\label{isotropy}
Let $G$ be a closed subgroup of $\Spin(7)$ and $X=G/H$ be an homogeneous not-associative three-submanifold of $S^7$. Then $\mathrm{dim}(H)\leq 3$. Moreover if $\mathrm{dim}(H)=1$ then $H^0$ must fix at least four directions in $\R^8$.
\end{lemma}
\begin{proof}
Let $x$ be the origin point of $X$ and $\rho:H\rightarrow T_xS^7$ be the isotropy representation. Clearly $\rho$ is injective and from Section \S\ref{sec1} it follows that $$\rho(H^0)\subseteq \mathrm{Stab}(\varphi_x)\cap\mathrm{Stab}(T_xX)^0\cong\SO(3),$$
where $\SO(3)$ acts on $T_xS^7$ as $\R^3\oplus\R^3\oplus\R$. Then the Lemma follows.  
\end{proof}
Now we are ready to prove the main theorem of this section.
\begin{proof}[Proof of \eqref{homogphiman}]
Let $G$ be a closed subgroup of $\Spin(7)$ whose three-dimensional $\Z_2-$invariant orbit $X$ trough $x\in S^7$ satisfies $\varphi|_X=0$ and let $H$ be the stabiliser of $x$ in $G$. Let us denote the Lie algebras of $G$ and $H$ by $\gt$ and $\hg$ respectively.\par 
First suppose that $\mathrm{dim}(\gt)=3$. Then by Proposition \eqref{subgroups} $G$ belongs to one of seven conjugacy classes. If $X/\Z_2$ is not-orientable, by compactness, its first Betti number cannot vanish, as well as that of $G$. Therefore $G$ would be conjugated to the assigned maximal torus $\U(1)^3$ of $\Spin(7)$; so let us assume they are equal. It is easy to see that $\Z_2$ acts trivially on its Lie algebra, and therefore any $\U(1)^3-$orbit in $\RP^7$ is orientable.\par
Suppose $\mathrm{dim}(\gt)>3$. Lemma \eqref{isotropy} implies that $1\leq\mathrm{dim}(\hg)\leq 3$. More precisely either $\mathrm{dim}(\hg)=1$ or $\mathrm{dim}(\hg)=3$; indeed $\SO(3)$ has rank one. In the second case $X$ should be diffeomorphic to $S^3$ and thus have vanishing first Betti number; in particular $X/\Z_2$ will be orientable. Therefore assume $\mathrm{dim}(\hg)=1$ and consequently $\mathrm{dim}(\gt)=4$.\par
Let $G_s\ltimes T$ be the Levi decomposition of $G$, where $G_s$ is a semisimple closed subgroup, with Lie algebra $\gt_s$, and $T$ a torus, with Lie algebra $\tg$, lying in its centralizer in $\Spin(7)$. Since $\Spin(7)$ has rank three $G_s$ must be three-dimensional and therefore conjugated to one of the simple Lie groups listed in Proposition \eqref{subgroups} different from (7).\par
Let $\hg_s$ and $\hg_t$ be the projections of $\hg$ on $\gt_s$ and $\tg$ respectively. If $\hg_t$ was one-dimensional then $\left\{A.x\;|\;A\in\gt_s\right\}$ would generate all $T_xX$ and consequently $X$ should be $G_s-$homogeneous as well, and thus described by one of the previous cases. Therefore we can suppose that $\hg\subset\gt_s$. \par
Then it is easy to see that the only two simple subgroups which admit non-trivial stabilisers are (5) and (6). But one can observe that in (5) any one-dimensional subspace generates a connected subgroup which fixes two lines only, in contrast with Lemma \eqref{isotropy}. Therefore the only admissible groups are conjugated to (6).\par
Let $(x_0,\dots,x_7)$ be standard coordinates on $\R^8$. With no loss of generality we can assume that $G_s$ acts on $x_0=x_7=0$: in fact $\Spin(7)$ is transitive on the Grassmannian of three-planes. Then a generic element $A\in \gt$ can be written as
$$ A=\sum_{j=0}^3 a_jE_j=
\left( \begin {array}{cccccccc} 0&0&0&0&0&0&0&3\,a_{{0}}
\\ \noalign{\medskip}0&0&a_{{1}}&a_{{2}}&0&0&-a_{{0}}&0
\\ \noalign{\medskip}0&-a_{{1}}&0&-a_{{3}}&0&a_{{0}}&0&0
\\ \noalign{\medskip}0&-a_{{2}}&a_{{3}}&0&a_{{0}}&0&0&0
\\ \noalign{\medskip}0&0&0&-a_{{0}}&0&a_{{3}}&a_{{2}}&0
\\ \noalign{\medskip}0&0&-a_{{0}}&0&-a_{{3}}&0&a_{{1}}&0
\\ \noalign{\medskip}0&a_{{0}}&0&0&-a_{{2}}&-a_{{1}}&0&0
\\ \noalign{\medskip}-3\,a_{{0}}&0&0&0&0&0&0&0\end {array} \right),
$$
where $E_0$ and $E_1,E_2,E_3$ span $\tg$ and $\gt_s$ respectively. Observe that $T$ acts as a diagonal $\SO(2)$ on $(x_0,x_7,x_6,x_1,x_2,x_5,x_4,x_3)$; in particular $\Z_2\subset T$ and $G_s\cap T=\left\{1\right\}$, whereas $G_s$ acts as a diagonal $\SO(3)$ on $(x_1,x_2,x_3,-x_6,x_4,x_5)$. Therefore $x$ must satisfy $(x_1,x_2,x_3)\wedge(-x_6,x_4,x_5)=0$ to be fixed by some non-trivial subgroup, and $x_0^2+x_7^2<1$ to have three-dimensional orbit. Clearly, acting with $G$, we can ensure that $x=(x_0,x_1,0,0,0,0,x_6,0)$.\par
Since the tangent space to the orbit at $x$ is three-dimensional and generated by 
$$v_1=(8x_0^2+1)^{-\frac{1}{2}}E_0x,\;v_2=(x_1^2+x_6^2)^{-\frac{1}{2}}E_1x,\;v_3=(x_1^2+x_6^2)^{-\frac{1}{2}}E_2x,$$
it turns out that
$$|\varphi_x(v_1,v_2,v_3)|=|\Phi_0(x,v_1,v_2,v_3)|=4\,{\frac {|x_{{0}}x_{{6}} \left( 3\,{x_{{1}}}^{2}-{x_{{6}}}^{2}
 \right) |}{ \left( {x_{{1}}}^{2}+{x_{{6}}}^{2} \right) \sqrt {8\,{x_{{0
}}}^{2}+1}}}
.$$
We therefore see that the homogeneous $\varphi-$manifolds are those given by either $x_0x_6=0$ or $3x_1^2=x_6^2$. Let us describe the topology of these orbits. Consider a generic element $at\in H$, where $a\in G_s$ and $t\in T$. Then
$$x_0=(atx)_0=(tx)_0=cx_0,\quad c\in[-1,1].$$
Therefore, if $x_0\neq 0$, $t$ equals the identity. Thus $H=H^0\cong\SO(2)$. In this case the orbit is isomorphic to $(\SO(3)/\SO(2))\times T\cong S^2\times S^1$ and the action of $\Z_2$ on $X$ equals the antipodal map on $S^1$. Otherwise, if $x_0=0$, from $atx=x$ it follows that
$$a=\left(\begin{matrix}
a_{11}&0&0\\
0&a_{22}&a_{23}\\
0&a_{32}&a_{33}\\
\end{matrix}\right),\;t=\left(\begin{matrix}
c&-s\\
s&c
\end{matrix}\right),\; a_{11}c=1.
$$
Consequently
$$H=\mathrm{O}(2)\tilde{\times}\Z_2=(\mathrm{O}(2)\times\Z_2)/((1_{\mathrm{O}(2)},1_{\Z_2})\sim(-1_{\mathrm{O}(2)},-1_{\Z_2}))\subset\SO(3)\times T,$$
and the action of $\Z_2$ on the orbit is given by
$$G/H\ni g H\mapsto (-1_{\Z_2}) g H\in G/H$$
Thus
$$\left(G/H\right)/\Z^2=\left(\SO(3)\times T\right)/\left(\mathrm{O}(2)\times \Z_2\right)\cong\left(\SO(3)/\mathrm{O}(2)\right)\times\left(T/ \Z_2\right)\cong \RP^2\times S^1.$$
Acting with $T$ we can ensure that $x_6=0$ and then see that the orbit is unique, passing through $x=(0,1,0,0,0,0,0,0)$ and defined by the intersection of the totally geodesic $\RP^5$, of equations $x_0=x_7=0$, and three quadrics
$$\left\{x_1x_5+x_2x_6=0\right\}\cap\left\{x_1x_4+x_3x_6=0\right\}\cap\left\{x_2x_4-x_3x_5=0\right\}.$$
\end{proof}

\section*{Acknowledgements}
The author is grateful to Prof. F. Podestà for all the useful suggestions and conversations he shared with him, to Prof. A. Fino for her constant interest in his work and finally to ‘‘Università degli studi di Firenze" for all the support he received.

\end{document}